\documentclass[a4paper,12pt, reqno]{amsart}
\usepackage{amssymb}
\usepackage{cite}
\usepackage{ifthen}
\usepackage[dvips]{graphicx}
\nonstopmode \numberwithin{equation}{section}
\newtheorem*{theoA}{Theorem A}
\newtheorem*{theoB}{Theorem B}

\newtheorem*{theoE}{Theorem E}

\usepackage{amssymb}
\usepackage{ifthen}
\usepackage{graphicx}
\usepackage{amsmath}
\usepackage[T1]{fontenc} 
\usepackage[utf8]{inputenc}
\usepackage[usenames,dvipsnames]{color}
\usepackage{color}
\usepackage[english]{babel}
\usepackage{fancyhdr}
\usepackage{fancybox}
\usepackage{tikz}

\theoremstyle{plain}
\newtheorem{prop}{Proposition}

\newtheorem{conj}{Conjecture}

\theoremstyle{definition}

\newtheorem{cor}{Corollary}[section]
\newtheorem{thm}{Theorem}[section]

\newtheorem{lem}{Lemma}[section]
\newtheorem{prob}{Problem}
\newtheorem{rem}{Remark}[section]

\theoremstyle{plain}

\newtheorem*{thmC}{Theorem C}
\newtheorem*{thmD}{Theorem D}

\newtheorem*{thmF}{Theorem F}

\newtheorem*{lemA}{Lemma A}
\newtheorem*{lemB}{Lemma B}
\newtheorem*{lemC}{Lemma C}

\newcounter{minutes}
\divide\time by 60
\newcounter{hours}
\multiply\time by 60
\addtocounter{minutes}{-\time}

\newcounter {own}
\def\theown {\thesection       .\arabic{own}}

\newenvironment{pf}[1][]{%
	\vskip 3mm
	\noindent
	\ifthenelse{\equal{#1}{}}%
	{{\slshape Proof. }}%
	{{\slshape #1.} }%
}%
{\qed\bigskip}

\newcounter{alphabet}

\def\be{\begin{equation}}
	\def\ee{\end{equation}}

\newcommand{\bee}{\begin{enumerate}}
	\newcommand{\eee}{\end{enumerate}}

\newcommand{\blem}{\begin{lem}}
	\newcommand{\elem}{\end{lem}}
\newcommand{\bthm}{\begin{thm}}
	\newcommand{\ethm}{\end{thm}}
\newcommand{\bcor}{\begin{cor}}
	\newcommand{\ecor}{\end{cor}}
\newcommand{\beg}{\begin{examp}}
	\newcommand{\eeg}{\end{examp}}
\newcommand{\begs}{\begin{examples}}
	\newcommand{\eegs}{\end{examples}}

\newcommand{\bdefn}{\begin{defn}}
	\newcommand{\edefn}{\end{defn}}

\newcommand{\bprob}{\begin{prob}}
	\newcommand{\eprob}{\end{prob}}
\newcommand{\bei}{\begin{itemize}}
	\newcommand{\eei}{\end{itemize}}

\newcommand{\bcon}{\begin{conj}}
	\newcommand{\econ}{\end{conj}}
\newcommand{\bcons}{\begin{conjs}}
	\newcommand{\econs}{\end{conjs}}
\newcommand{\bprop}{\begin{prop}}
	\newcommand{\eprop}{\end{prop}}
\newcommand{\br}{\begin{rem}}
	\newcommand{\er}{\end{rem}}
\newcommand{\brs}{\begin{rems}}
	\newcommand{\ers}{\end{rems}}
\newcommand{\bo}{\begin{obser}}
	\newcommand{\eo}{\end{obser}}
\newcommand{\bos}{\begin{obsers}}
	\newcommand{\eos}{\end{obsers}}
\newcommand{\bpf}{\begin{pf}}
	\newcommand{\epf}{\end{pf}}
\newcommand{\ba}{\begin{array}}
	\newcommand{\ea}{\end{array}}
\newcommand{\beq}{\begin{eqnarray}}
	\newcommand{\beqq}{\begin{eqnarray*}}
		\newcommand{\eeq}{\end{eqnarray}}
	\newcommand{\eeqq}{\end{eqnarray*}}

\makeatletter
\@namedef{subjclassname@2020}{%
  \textup{2020} Mathematics Subject Classification}
\makeatother

\begin{document}
\title{The Bohr inequality for vector-valued holomorphic functions with lacunary series in complex Banach spaces}

\author{Sabir Ahammed}
\address{Sabir Ahammed, Department of Mathematics, Jadavpur University, Kolkata-700032, West Bengal, India.}
\email{sabira.math.rs@jadavpuruniversity.in}

\author{Molla Basir Ahamed}
\address{Molla Basir Ahamed, Department of Mathematics, Jadavpur University, Kolkata-700032, West Bengal, India.}
\email{mbahamed.math@jadavpuruniversity.in,rageshh.math.rs@jadavpuruniversity.in}

\author{Rajesh Hossain}
\address{Rajesh Hossain, Department of Mathematics, Jadavpur University, Kolkata-700032, West Bengal, India.}
\email{rajesh1998hossain@gmail.com}


\def\thefootnote{}
\footnotetext{ {\tiny File:~\jobname.tex,
printed: \number\year-\number\month-\number\day,
          \thehours.\ifnum\theminutes<10{0}\fi\theminutes }
} \makeatletter\def\thefootnote{\@arabic\c@footnote}\makeatother

\begin{abstract} 
In this paper, we study the Bohr inequality with lacunary series for vector-valued holomorphic functions defined in unit ball of finite dimensional Banach sequence space. Also, we study the Bohr-Rogosinski inequality for same class of functions. All the results are proved to be sharp.
\end{abstract}

\maketitle
\pagestyle{myheadings}
\markboth{S. Ahammed,  M. B. Ahamed and R. Hossain }{The Bohr inequality for vector-valued holomorphic functions }
\tableofcontents
\section{Introduction}
\label{intro}
\subsection{The classical Bohr inequality and its recent implications }
Let $H^\infty$ denote the class of all bounded analytic functions $f$ in the unit disk $\mathbb{D}:=\{z \in \mathbb{C} : |z| < 1\}$ equipped with the topology of uniform convergence on compact subsets of $\mathbb{D}$ with the supremum norm $\| f \|_\infty := \sup_{z \in \mathbb{D}} | f (z)|$ and $\mathcal{B}:= \{ f \in H^\infty : \| f \|_\infty \leq 1\}$. Let us start with a remarkable result of Harald Bohr published in $ 1914 $, dealing with a problem connected with Dirichlet series and number theory, which stimulated a lot of research activity into geometric function theory in recent years.
\begin{theoA}\cite{Bohr-1914}
	If $ f(z)=\sum_{s=0}^{\infty}a_sz^s\in\mathcal{B} $, then 
	\begin{equation}\label{e-1.2}
		\sum_{s=0}^{\infty}|a_s|r^s\leq 1 \;\; \mbox{for}\;\; |z|=r\leq {1}/{3}.
	\end{equation}
\end{theoA}
The inequality fails when $ r>{1}/{3} $ in the sense that there are functions in $ \mathcal{B} $ for which the inequality is reversed when $ r>{1}/{3} $. H. Bohr initially showed that the inequality \eqref{e-1.2} holds only for $|z|\leq1/6$, which was later improved independently by M. Riesz, I. Schur, F. Wiener and some others. The sharp constant $1/3$  and the inequality \eqref{e-1.2} in Theorem A are called respectively, the Bohr radius and the  classical Bohr inequality for the family $\mathcal{B}$.  
A direct proof of it with the help of Rogosinski's coefficient inequality for function subordinate to
a univalent function has been indicated in \cite{Ponnusamy-Wirths-CMFT-2020} 
which motivates to extend many results.
Several other proofs of this interesting inequality were given in different articles
(see \cite{Sidon-1927,Paulsen-Popescu-Singh-PLMS-2002,Tomic-1962}).\vspace{1.2mm}

\subsection{Basic Notations}
For $m \in \mathbb{N} := \{1, 2, \dots\}$, let
\begin{align*}
	\mathcal{B}_m = \{\omega \in B : \omega(0) = \dots = \omega^{(m-1)}(0) = 0 \text{ and } \omega^{(m)}(0) \neq 0\}
\end{align*}
so that $\mathcal{B}_1 = \{\omega \in B : \omega(0) = 0 \text{ and } \omega'(0) \neq 0\}$. Also, for $f (z) = \sum_{s=0}^\infty a_s z^s \in \mathcal{B}$
and $f_0(z) := f (z) - f (0)$, we let (as in \cite{Ponnusamy-Vijayakumar-Wirths-HJM-2021})

\begin{align*}
	B_N(f,r) := \sum_{s=N}^\infty |a_s|r^s\;\;\mbox{for}\;\;N \geq 0\;\; \mbox{and}\;\; \| f_0\|_2^2 r := \sum_{s=1}^\infty |a_s|^2 r^{2s},
\end{align*}
and in what follows we introduce
\begin{align*}
	A( f_0, r ) := \left( \frac{1}{1 + |a_0|} + \frac{r}{1 - r} \right) \| f_0\|_2^2 r,
\end{align*}
which helps to reformulate refined classical Bohr inequalities.
\subsection{Recent Bohr-type inequalities}
In recent years, the study of Bohr phenomena have been an active research topic. Many researchers continuously investigated the Bohr-type inequalities and also examining their sharpness for certain classes of analytic functions. In this follow, Kayumov and Ponnusamy established the following result.
\begin{theoB}\cite[Theorem 3]{Kayumov-Ponnusamy-CRACAD-2018}
	Suppose that $f(z)=\sum_{s=0}^{\infty}a_sz^s\in\mathcal{B} $. Then
	\begin{equation*}
			B_0(f,r) +|f_0(z)|^2\leq 1 \;\; \mbox{for}\;\; |z|=r\leq {1}/{3}.
	\end{equation*}
	The number $1/3$ cannot be improved.
\end{theoB}

Let $f$ be holomorphic in $\mathbb D$, and for $0<r<1$,  let $\mathbb D_r=\{z\in \mathbb C: |z|<r\}$.
Let $S_r:=S_r(f)$ denote the planar integral
\begin{equation*}
\label{Sr-Def}
	S_r=\int_{\mathbb D_r} |f'(z)|^2 d A(z).
\end{equation*}
If the function $f\in \mathcal{B}$ has Taylor's series expansion $f(z)=\sum_{s=0}^{\infty}a_sz^s $, then we obtain  (see \cite{Kayumov-Ponnusamy-CRACAD-2018})
\begin{align*}
	S_r= \pi\sum_{s=1}^\infty s|a_s|^2 r^{2s}.
\end{align*}
  
In the study of the improved Bohr inequality, the quantity $ S_r $ plays a significant role. There are many results on the improved Bohr inequality for the class $ \mathcal{B} $ (see \cite{Ismagilov- Kayumov- Ponnusamy-2020-JMAA,Kayumov-Ponnusamy-CRACAD-2018}), and for harmonic mappings on unit disk (see \cite{Evdoridis-Ponnusamy-Rasila-Indag.Math.-2019}). Liu \emph{et al.} further studied the improved Bohr inequality with proper combinations of $S_r/\pi$ and obtain the following result.
\begin{thmC}\cite[Theorem 4]{Liu-Liu-Ponnusamy-BSM-2021}
	Suppose that $f(z)=\sum_{s=0}^{\infty}a_sz^s\in\mathcal{B} $. Then
	\begin{align*}
		B_0(f,r)+A(f_0,r)+\dfrac{8}{9}\left(\dfrac{S_r}{\pi}\right)\leq 1 \;\;\mbox{for}\;\; |z|=r\leq \dfrac{1}{3}.
	\end{align*}
	The constant $1/3$ cannot be improved.
\end{thmC}
Jia \emph{et al.} \cite{Jia-Liu-Ponnusamy-AMP-2025} further extended the classical Bohr inequality and obtain the following sharp result.
\begin{thmD}\cite[Theorem 1]{Jia-Liu-Ponnusamy-AMP-2025}
	Suppose that $f\in \mathcal{B}$ has the expansion $f(z)=\sum_{s=0}^\infty a_s z^s$ with $|f(0)|<1$, $\omega_s \in \mathcal{B}_s$ for $s \geq 1$ and $ R_{p}^{k,m}$ is the unique root in $(0, 1)$ of the equation 
	\begin{align*}
	 \frac{r^k}{1 - r^k} + \frac{r^m}{1 - r^m} - \frac{p}{2}=0,
	\end{align*}
	for some $k,m \in \mathbb{N}$ and $p\in(0, 2]$. Then we have
	\begin{align*}
		|f(0)|^p + B_1(f , |\omega_k (z)|) + A(f_0, |\omega_k (z)|) + |f (\omega_m(z)) - f(0)| \leq 1
	\end{align*}
	for $|z| = r \leq  R_{p}^{k,m}.$ The number $ R_{p}^{k,m}$ cannot be improved.
\end{thmD}



 In order to determine the Bohr radius for the class of odd functions in the family $ \mathcal{B} $, which was posed in \cite{Ali-2017}, Kayumov and Ponnusamy 
 \cite{Kayumov-Ponnusamy-CMFT-2017,Kayumov-Ponnusamy-2018-JMAA}
 studied  the Bohr inequalities for holomorphic functions in a single complex variable.
 Generalizations of this result to holomorphic mappings in several complex variables have been studied (see e.g. 
 \cite{Arora-CVEE-2022,Lin-Liu-Ponnusamy-Acta-2023,Liu-Ponnusamy-PAMS-2021,Hamada-Honda-Kohr-AMP-2025,Hamada-Honda-RM-2024,Aha-Aha-Ham-AMC,Lata-Singh-PAMS-2022}).
The Bohr phenomenon has been extended to
holomorphic or pluriharmonic functions of several variables (see e.g. 
\cite{Boas-Khavinson-PAMS-1997,Liu-Ponnusamy-PAMS-2021,Hamada-Math.Nachr.-2021,Hamada- Honda-BMMS-2024,Aizenberg-PAMS-2000,Hamada-IJM-2009,Lin-Liu-Ponnusamy-Acta-2023,Kumar-PAMS-2022,Kumar-Manna-JMAA-2023,Kumar-Ponnusamy-Williams-NYJM-2025}).

\subsection{The Bohr-Rogosinski inequality and its recent improvements} Similar to the Bohr radius, the notion of the Rogosinski radius was first introduced in \cite{Rogosinski-1923} for functions
$f\in\mathcal{B}$. 
Nevertheless,  as compared to the Bohr radius, the Rogosinski radius has not received the same level of research attention. 
If $ B $ and $ R $
denote the Bohr radius and the Rogosinski radius, respectively, then it is easy to see that $ B = 1/3 < 1/2 = R $.

Moreover, analogous to the Bohr inequality, there is also a concept of the Bohr-Rogosinski inequality. Following the article \cite{Kayumov-Khammatova-Ponnusamy-JMAA-2021},
for the functions $ f(z)=\sum_{s=0}^{\infty}a_sz^s\in\mathcal{B}$, the Bohr-Rogosinski sum $ R^f_N(z) $ of $ f $ is defined by 
\begin{align}\label{e-0.2}
	R^f_N(z):=|f(z)|+\sum_{s=N}^{\infty}|a_s|r^s, \; |z|=r.
\end{align}
An interesting fact to be observed is that for $ N=1 $, the quantity in \eqref{e-0.2} is related to the classical Bohr sum in which $ |f(0)| $ is replaced by $ |f(z)| $. The relation $ R^f_N(z)\leq 1 $ is called the Bohr-Rogosinski inequality. 
For some recent development on the Bohr-Rogosinski inequality, the reader is refereed to the article  
\cite{ Allu-Arora-JMAA-2022,Das-JMAA-2022,Ponnusamy-Vijayakumar-Current Research-2022}, and the references therein. In recent times, the study of Bohr-Rogosinski radius for holomorphic mappings with values in higher dimensional complex Banach spaces is an active research area, and Hamada \emph{et al.} \cite{Hamada-Honda-Kohr-AMP-2025} studied the  Bohr–Rogosinski inequalities for holomorphic mappings with values in higher dimensional complex Banach spaces. Chen \emph{et al.} \cite{Chen-Liu-Ponnusamy-RM-2023} studied the following Bohr-type inequality for bounded analytic self-map on $\mathbb{D}.$
\begin{theoE} \cite[Theorem 6]{Chen-Liu-Ponnusamy-RM-2023}
Suppose that $ f(z)=\sum_{s=0}^{\infty}a_sz^s\in\mathcal{B},$	$p\in (0,2]$, $m,q\geq 2,$ $0<m<q$ and  let $v_m:\mathbb{D}\to \mathbb{D}$  be  {\it Schwarz mappings} having $z=0$ as a zero of order $m.$ For arbitrary $\lambda\in (0,\infty),$ we have 
\begin{align*}
	|f(v_m(z))|^p+\lambda\sum_{s=1}^{\infty}|a_{qs+m}|r^{qs+m}\leq 1\;\;\mbox{for}\;\; |z|=r\leq R_{q,m,\lambda}^p,
\end{align*}
where $R_{q,m,\lambda}^p$ is the minimal root in $(0,1)$ of equation 
\begin{align*}
\Psi(r):=2\lambda\dfrac{r^{q+m}}{1-r^q}-p\dfrac{1-r^m}{1+r^m}=0.
\end{align*}
In the case when $\Psi(r)>0$ in some interval $\left(R_{q,m,\lambda}^p,R_{q,m,\lambda}^p+\epsilon\right),$ the number $R_{q,m,\lambda}^p$ cannot be improved.
\end{theoE} 

In recent years, refining the Bohr-type inequalities
have been an active research topic. Many researchers continuously investigated refined Bohr-type inequalities and also examining their sharpness for certain classes of analytic functions, for classes of harmonic mappings on the unit disk $ \mathbb{D} $. For detailed information on such studies, the readers are referred to  
\cite{Ahamed-AASFM-2022,Evdoridis-Ponnusamy-Rasila-RM-2021,Liu-Ponnusamy-Wang-RACSAM-2020,Liu-Liu-Ponnusamy-BSM-2021,Ponnusamy-Vijayakumar-Wirths-JMAA-2022} and the references therein. To continue the study on the Bohr-type inequalities, for any $N\in\mathbb{N}$ and $ k=\lfloor{(N-1)/2}\rfloor,$ we define the following functional:
\begin{align*}
	\mathcal{Q}_{f,N}(r):=\sum_{s=N}^{\infty}|a_s|r^s+\mathrm{sgn}(k)\sum_{s=1}^{k}|a_s|^2\dfrac{r^N}{1-r}+\left(\dfrac{1}{1+|a_0|}+\dfrac{r}{1-r}\right)\sum_{s=k+1}^{\infty}|a_s|^2r^{2s}.
\end{align*}
 Liu \emph{et al.} \cite[Theorem 1]{Liu-Liu-Ponnusamy-BSM-2021} obtained the following refined version of the Bohr-Rogosinski  inequality for class of functions $ f\in\mathcal{B}.$
\begin{thmF}\cite[Theorem 1]{Liu-Liu-Ponnusamy-BSM-2021} Suppose that $ f(z)=\sum_{s=0}^{\infty}a_sz^s\in\mathcal{B}.$ Then
	\begin{align}\label{Eq-1.3}
		|f(z)|+\mathcal{Q}_{f,N}(r)\leq 1\;\;\mbox{for}\;\; |z|= r\leq R_N,
	\end{align}
	where $R_N$ is the positive root of the equation $2(1+r)r^N-(1-r)^2=0$ The radius $R_N$ is best possible. Moreover,
	\begin{align}\label{Eq-1.4}
		|f(z)|^2+\mathcal{Q}_{f,N}(r)\leq 1\;\;\mbox{for}\;\; |z|= r\leq R^\prime_N,
	\end{align}
		where $R^\prime_N$ is the positive root of the equation $(1+r)r^N-(1-r)^2=0$ The radius $R^\prime_N$  is best possible.
\end{thmF}

\subsection{New problems on multi-dimensional Bohr’s inequality} In the recent years, many authors paid attention to multidimensional generalizations of Bohr’s theorem and draw many conclusions. For example, denote an $n$-variables power series by $\sum_{\alpha}a_\alpha z^\alpha$ with the standard multi-index notation; $\alpha$ denotes an $n$-tuple $(\alpha_1,\alpha_2,\dots ,\alpha_n)$ of non-negative integers, $|\alpha|$ denotes the sum $\alpha_1+\alpha_2+\dots+ \alpha_n$ of its components, $\alpha!$ denotes the product of the factorials $\alpha_1!\alpha_2!\dots \alpha_n!$ of its components, $z=(z_1,\dots,z_n)\in \mathbb{C}^n,$ $z^\alpha=z^{\alpha_1}_1z^{\alpha_2}_2\dots z^{\alpha_n}_n.$ The $n$-dimensional Bohr radius $K_n$ is the largest number such that if $\sum_{\alpha}a_\alpha z^\alpha$ converges in the $n$-dimensional unit polydisk $\mathbb{D}^n$ such that 
\begin{align*}
	\bigg|\sum_{\alpha}a_\alpha z^\alpha\bigg|<1
\end{align*} in $\mathbb{D}^n,$ the $n$-dimensional Bohr radius $K_n$ satisfies
\begin{align*}
	\dfrac{1}{3\sqrt{n}}<K_n<2\sqrt{\dfrac{\log n}{n}}.
\end{align*}
This article became a source of inspiration for many subsequent investigations including connecting the asymptotic behaviour of $K_n$ to problems in the geometry of Banach spaces (cf. \cite{Defant-2019}). However determining the exact value of the Bohr radius $K_n$, $n>1$, remains an open problem. In $2006$, Defant and Frerick \cite{Defant-Frerick-IJM-2006} improved the lower bound as $K_n\geq c\sqrt{\log n/(n\log \log n)}$ whereas Defant \emph{et al.} \cite{Defant-Frerick-Ortega-Cerd-Seip-2011} used the hypercontractivity of the polynomial Bohnenblust-Hille inequality and showed that 
\begin{align*}
	K_n=b_n\sqrt{\dfrac{\log n}{n}}\;\;\mbox{with}\;\; \dfrac{1}{\sqrt{2}}+o(1)\leq b_n\leq 2. 
\end{align*}
In $2014,$ Bayart \emph{et al.} \cite{Bayart-Pellegrino-Seoane-Adv-2014} established the asymptotic behaviour of $K_n$ by showing that 
\begin{align*}
	\lim\limits_{n\to \infty}\dfrac{K_n}{\sqrt{\frac{\log n}{n}}}=1.
\end{align*}
Blasco \cite{Blasco-OTAA-2010} have studied the asymptotic behavior of the holomorphic functions with $p$-norm as $r\to 1$ in $\mathbb{D}^n$ and Banach spaces. Aizenberg \cite{Aizenberg-PAMS-2000,Aizenberg-OTAA-2005} mainly generalized Carath\'eodory’s inequality for functions holomorphic in $\mathbb{C}^n.$ In 2021, Liu and Ponnusamy \cite{Liu-Ponnusamy-PAMS-2021} have established several multidimensional analogues of refined Bohr’s inequality for holomorphic functions on complete circular domain in $\mathbb{C}^n.$  Other aspects and promotion of the Bohr inequality in higher dimensions can be obtained from \cite{Hamada-IJM-2009,Boas-Khavinson-PAMS-1997,Galicer-TAMS-2021,Defant-AM-2012-Advanc}. Moreover, research on Dirichlet series in higher dimensions is also very popular recently (see \cite{Defant-2019}).\vspace{2mm}

In 2020, Liu and Liu \cite{Liu-Liu-JMAA-2020} used the Fr\'echet derivative to establish the Bohr inequality of norm-type for holomorphic mappings with lacunary series on the unit polydisk in $\mathbb{C}^n$ under some restricted conditions. The relevant properties of the Fr\'echet derivative can be seen below (cf. \cite{Graham-2003}).
Throughout the paper, we denote the set of non-negative integers
by $\mathbb{N}_0$.\vspace{2mm}

Let $F:B_X\to Y$ be a holomorphic mapping.
For $k\in {\mathbb N}$, 
we say that $z=0$ is a {\it zero of order $k$ of $F$} if
$F(0)=0$, $DF(0)=0$, $\dots$, $D^{k-1}F(0)=0$,
but $D^kF(0)\neq 0$.\vspace{2mm}

A holomorphic mapping $v:B_X\to B_Y$ with $v(0)=0$
is called a {\it Schwarz mapping}.
We note that if $v$ is a Schwarz mapping such that $z=0$ is a zero of order $k$ of $v$,
then the following estimation holds (see e.g. \cite[Lemma 6.1.28]{Graham-2003}):
\begin{align}\label{Schwarz-k}
	\| v(z)\|_Y\leq \| z\|_X^k,
	\quad z\in B_X.
\end{align}

Let $n\in \mathbb{N},$ $t\in [1,\infty),$ and $B_{{\ell}_t^n}$ be the set defined as the collection of complex vectors $z=(z_1,z_2,\dots, z_n)\in \mathbb{C}^n$ satisfying $\sum_{j=1}^{n}|z_j|^t<1.$ This set constitutes the open unit ball in the complex Banach space ${\ell}_t^n$ where the norm $||z||_t$ of $z$ is given by  $\left(\sum_{j=1}^{n}|z_j|^t\right)^{1/t}<1.$ In the special case of $B_{{\ell}_t^n},$  the set represents the unit polydisk in $\mathbb{C}^n$ denoted as $B_{{\ell}_t^n}:=\mathbb{D}^n,$ where $|z_j|<1$ for $1\leq j\leq n.$ The norm of $z\in {\ell}_t^n$ is defined as $||z||:=\max{|z_j|:1\leq j\leq n}.$ Note that the unit disk $\mathbb{D}$ is equivalent to $B_{{\ell}_t^n}.$\vspace{2mm}

It is natural to raise the following problems.
\begin{prob}\label{P-1}
	Can we establish Theorems B, D, E, and F for vector-valued holomorphic functions with lacunary series from $B_{{\ell}_t^n}$ to $\mathbb{D}^n $ involving Schwarz mappings?
\end{prob}
\begin{prob}\label{P-2}
	Can we establish Theorem C for vector-valued holomorphic functions  with lacunary series from $B_{{\ell}_t^n}$ to $\mathbb{D}^n ?$ 
\end{prob}
In this article, we aim to provide affirmative answers to Problems \ref{P-1} and \ref{P-2}. We begin in Section 2 by presenting our theorems and several key remarks. It's worth noting that Theorems \ref{Thm-4.1}, \ref{Thm-1.2} and \ref{Thm-1.4} primarily offer an affirmative answer to Problem 1, and Theorem   \ref{Thm-2.1} address Problem 2. Following this, Section 3 provides the necessary lemmas that underpin the proofs of our theorems. All theorems are then fully proven in Section 4.
\section{Bohr-type inequalities for vector-valued holomorphic mappings with lacunary series in complex Banach spaces}
Let $X$ and $Y$ be  complex Banach spaces with norms $||\cdot||_X$
and $||\cdot||_Y$, respectively.
For simplicity, we omit the subscript for the norm when it is obvious from the context. 
Let $B_X$ and $B_Y$ be the open unit balls in $X$ and $Y$, respectively. 
If $X=\mathbb{C},$ then $B_X=\mathbb{D}$
is the unit disk in $\mathbb{C}$. Let $\mathcal{H}\left(\Omega,Y\right)$ denote the set of all holomorphic mappings from $\Omega$ into $Y$.\vspace{2mm}

For  $F\in \mathcal{H}\left(B_X,Y\right)$ and $z\in B_X$, let $D^kF(z)$
denote the $k$-th Fr\'{e}chet derivative of $F$ at $z$.
It is well-known (cf. \cite{Graham-2003}) that for any holomorphic mapping  $F\in \mathcal{H}\left(B_X,Y\right)$
can be expanded into the series
\begin{equation}
	\label{expansion}
	F(z) =\sum_{s=0}^{\infty}\frac{D^s F(0)(z^s)}{s!}
\end{equation}
for all $z$ in some neighbourhood of $0\in B_X$, where $D^kF(z)$ is the $k$-th Fr\'echet derivative of $F$ at $z$ and for each $k\in \mathbb{N},$ we have 
\begin{align*}
	D^k f(0)(z^k) = D^k f(0)(\underbrace{z, z, \dots, z}_{k}).
\end{align*}
Moreover, if $k=0,$ then $D^0F(0)(z^0)=F(0).$
Note that if $F(B_X)$ is bounded,
then (\ref{expansion}) converges uniformly on $rB_X$ for each $r\in (0,1)$.\vspace{2mm} 

For each $x\in X\setminus\{0\},$ we define
\begin{align*}
	T(x)=\{T_x\in X^*:||T_x||=1, T_x(x)=||x||\}, 
\end{align*}
where
$X^*$ is the dual space of $X$.
Then the well known Hahn-Banach theorem implies that $T(x)$ is non empty.\vspace{2mm}

In this section, we will extend the Bohr-type inequalities to higher dimensional spaces using the Fr\'echet derivative.
\subsection{Extension of Theorems B, D, E and F for functions in the class $\mathcal{H}\left(B_{{\ell}_t^n},\overline{\mathbb{D}}^n \right)$}

We obtain the following result for vector-valued holomorphic functions defined in the unit ball of a finite-dimensional Banach sequence space. The inequality we consider here combines versions of those in Theorems B and E.
\begin{thm}\label{Thm-4.1}
Suppose that $1\leq t\leq \infty,$ and  $f\in \mathcal{H}\left(B_{{\ell}_t^n},\overline{\mathbb{D}}^n \right)$  with series expansion 	
\begin{align}\label{Eq-BS-1.6}
	f(z)=\sum_{s=0}^{\infty}\dfrac{D^sf(0)(z^s)}{s!},\;\; z\in B_{{\ell}_t^n},
\end{align}
where $D^0f(0)(z^0)=f(0)=a=(a_1,\dots, a_n)$ with $|a_j|=||a||_{\infty}$ for all $j\in \{ 1,2,\dots, n\}.$  Let $v_1, v_2:B_{{\ell}_t^n}\to B_{{\ell}_t^n}$  be  {\it Schwarz mappings} having $z=0$ as a zero of order $m_1, m_2$,
	respectively.
	Then, for $p\in (0,2],$ $m\in \mathbb{N}_0$, $q\in \mathbb{N}$ and $\mu, \nu \in [0,\infty)$ with $\mu+\nu>0$, we have
	\begin{align*}
		&{\|f\left(v_1(z)\right)\|}^p+\mu\sum_{s=1}^{\infty}\dfrac{\|D^{qs+m}f(0)(z^{qs+m})\|_\infty}{{(qs+m)!}}
		+\nu \|f(v_2(z))-f(0)\|_\infty\leq 1
	\end{align*} 
	for $ ||z||=r\leq R^{m,p}_{\mu,\nu}(m_1,m_2):=R_1,$	where
	$R_1$ is the minimal root in $(0,1)$ of the equation 
	\begin{align}\label{Eq-4.1}
		\Xi(r):=2\mu \frac{r^{q+m}}{1-r^q}+2\nu \frac{r^{m_2}}{1-r^{m_2}}-p\left(\frac{1-r^{m_1}}{1+r^{m_1}}\right)=0.
	\end{align}
	The constant $R_1$ cannot be improved.
\end{thm}
	If we set $\mu=\nu=m_2=q=1=p$, $m=0$ and $m_1\to \infty$ in  Theorem \ref{Thm-4.1}, then we obtain the following corollary which improved the classical Bohr inequality.
\begin{cor}
Suppose that $f\in \mathcal{H}\left(B_{{\ell}_t^n},\overline{\mathbb{D}}^n \right)$ be same as in Theorem \ref{Thm-4.1}. Then 
\begin{align*}
\sum_{s=0}^{\infty}\dfrac{\|D^{s}f(0)(z^{s})\|_\infty}{{s!}}
	+\|f(z)-f(0)\|_\infty\leq 1\;\;\mbox{for}\;\; r\leq \dfrac{1}{5}
\end{align*} 
and the constant $1/5$ cannot be improved.
\end{cor}

\begin{rem}
If we set $\mu=\nu=m_2=1$, $m_1=m$  in  Theorem \ref{Thm-4.1}, then we obtain the following corollary which is an analogue of Theorem E.
\end{rem}

Recently,  Liu \cite{Liu-Liu-Ponnusamy-BSM-2021} introduced a new refined versions of the classical Bohr inequality and obtained several new sharp results. For  further study of Bohr-type inequalities, we define a function as follows.
\begin{align*}
\mathcal{N}_{f}^{1}({\|z\|_t}):=\sum_{s=1}^{\infty}\dfrac{{\|D^sf(0)(z^s)\|}_\infty}{s!}+\left(\frac{1}{1+{{\|f(0)\|}_\infty}}+\frac{{\|z\|}_t}{1-{\|z\|}_t}\right)\sum_{s=1}^{\infty}\left(\dfrac{{\|D^sf(0)(z^s)\|}_\infty}{s!}\right)^2.
\end{align*} 

In our next result, we obtain an extended version of Theorem D for functions  $f\in \mathcal{H}\left(B_{{\ell}_t^n},\overline{\mathbb{D}}^n \right)$ involving Schwarz mappings $v_1:B_{{\ell}_t^n}\to B_{{\ell}_t^n}$ having $z=0$ as a zero of order $m_1$.
\begin{thm}\label{Thm-1.2}
Suppose that $1\leq t\leq \infty,$ and  $f\in \mathcal{H}\left(B_{{\ell}_t^n},\overline{\mathbb{D}}^n \right)$  with series expansion given by \eqref{Eq-BS-1.6}. Let $v_1:B_{{\ell}_t^n}\to B_{{\ell}_t^n}$  be a {\it Schwarz mappings} having $z=0$ as a zero of order $m_1$. For $p\in (0,2],$ we have 
\begin{align*}
	\mathcal{B}^p_f(r):={\|f(0)\|}_\infty^p+\mathcal{N}_{f}^{1}({\|z\|_t})+{\|f(v_1(z))-f(0)\|}_\infty\leq 1\;\;\mbox{for}\;\; \|z\|_t=r\leq R_2(p),
\end{align*}
	where $R_2(p)$ is the minimal root in $(0,1)$ of the equation 
	\begin{align*}
		\dfrac{r}{1-r}+\dfrac{r^{m_1}}{1-r^{m_1}}- \dfrac{p}{2}=0.
	\end{align*}
The constant $R_2(p)$ cannot be improved.
\end{thm}

Similar to the quantity $S_r$ for functions $f\in \mathcal{B},$ we define $S_z$ for holomorphic functions $F\in \mathcal{H}\left(B_{{\ell}_t^n},\overline{\mathbb{D}}^n \right)$  with series expansion given by \eqref{Eq-BS-1.6} as follows: 
\begin{align}\label{BS-eq-1.3}
S_z:=\sum_{s=1}^{\infty}s \left(\dfrac{{\|D^sF(0)(z^s)\|}_\infty}{s!}\right)^2.
\end{align}
Our aim is to establish Theorem C with a more general setting for vector-valued holomorphic functions with lacunary series from $B_{{\ell}_t^n}$ to $\mathbb{D}^n$. To this end, we consider a polynomial in $x$ of degree $N$ as follows: 
\begin{align}\label{BS-eq-2.77}
	W_N(x) := d_1x + d_2x^2 + \dots + d_Nx^N, \quad \text{where } d_i \geq 0 \text{ for } i=1, 2, \dots, N,
\end{align}
and obtain the following result.
\begin{thm}\label{Thm-2.1}
Suppose that $1\leq t\leq \infty,$ and  $f\in \mathcal{H}\left(B_{{\ell}_t^n},\overline{\mathbb{D}}^n \right)$  with series expansion 	
with series expansion given by \eqref{Eq-BS-1.6}.  For $p\in (0,1],$ we have 
	\begin{align*}
		\mathcal{C}_f(r):={\|f(0)\|}^p_\infty+&\mathcal{N}_{f}^{1}({\|z\|_t})+W_N\left(S_z\right)\leq 1, 
	\end{align*}
	for ${\|z\|}_t|=r\leq R_3(p):={p}/{(2+p)},$
	where  the coefficient of the polynomial $W_N$ satisfy the condition 
	\begin{align}\label{eq-2.2}
		8d_1M^2_p+6c_2d_2M^4_p+\dots+2(2N-1)C_NM^{2N}_p\leq p,
	\end{align}
	with $M_p:=p(2+p)/(4p+4)$ and 
	\begin{align*}
		c_s:=\max_{t\in [0,1]}\left(t(1+t)^2(1-t^2)^{2s-2}\right), s=2,\dots N.
	\end{align*} 
	The constant $R_3(p)$ cannot be improved for each $p\in (0,1]$ and for each $d_1, \dots d_N$ which satisfy \eqref{eq-2.2}.
\end{thm}
\begin{rem}
	In particular, if we set $p=N=1,$ $d_1=8/9$ and $d_j=0$ for $j=2,3,\dots, N,$ in Theorem \ref{Thm-2.1}, then we obtain the following corollary which is an extension of Theorem C for vector-valued holomorphic functions with lacunary series from $B_{{\ell}_t^n}$ to $\mathbb{D}^n$. The corollary is presented below.
\end{rem}

\begin{rem}\label{Rem-2.1}
	Ismagilov \emph{et al.} \cite{Ismagilov- Kayumov- Ponnusamy-2020-JMAA} remarked that for any function $F : [0, \infty)\to [0, \infty)$ such that $F(t)>0$ for $t>0$, there exists an analytic function $f : \mathbb{D}\to\mathbb{D}$ for which the inequality
	\begin{align}\label{Eq-22.77}\sum_{s=0}^{\infty}|a_s|r^s+\frac{16}{9}\left(\frac{S_r}{\pi}\right)+\lambda\left(\frac{S_r}{\pi}\right)^2+F(S_r)\leq 1\;\; \mbox{for}\;\; r\leq\frac{1}{3}
	\end{align} is false, where $S_r$ is given in (\ref{Sr-Def}) and $\lambda$ is given in \cite[Theorem 1]{Ismagilov- Kayumov- Ponnusamy-2020-JMAA}. However, it is worth noting that, by defining $F(S_z)=d_3(S_z)^3+\cdots+d_N(S_z)^N>0$, one can observe from Theorem \ref{Thm-2.1} that inequality \eqref{Eq-22.77} holds when  $f\in \mathcal{H}\left(B_{{\ell}_t^n},\overline{\mathbb{D}}^n \right)$  with series expansion with series expansion given by \eqref{Eq-BS-1.6}.
	\end{rem}

\begin{cor} Suppose that $1\leq t\leq \infty,$ and  $f\in \mathcal{H}\left(B_{{\ell}_t^n},\overline{\mathbb{D}}^n \right)$  with series expansion given by \eqref{Eq-BS-1.6}.  Then 
	\begin{align*}
		{\|f(0)\|_\infty}+\mathcal{N}_{f}^{1}({\|z\|_t})+\dfrac{8}{9}S_z^*\leq 1
	\end{align*}
	for $ {\|z\|}_t=r\leq {1}/{3}.$ The constant $ 1/3$ is best possible.
\end{cor}

In the next result, we establish the Bohr-Rogosinski inequality for vector-valued holomorphic functions with lacunary series from $B_{{\ell}_t^n}$ to $\mathbb{D}^n$.
\begin{thm}\label{Thm-1.4}
Suppose that $1\leq t\leq \infty,$ and  $f\in \mathcal{H}\left(B_{{\ell}_t^n},\overline{\mathbb{D}}^n \right)$  with series expansion given by \eqref{Eq-BS-1.6}.  For $p\in (0,2],$ we have 	
\begin{align}\label{Eq-2.6}
	\mathcal{D}_f(r):={\|f(v_1(z))\|}_\infty^p+\mathcal{M}_{f}^{N}({\|z\|_t})\leq 1\;\;\mbox{for}\;\; \|z\|_t=r\leq R^p_{m_1,N}
\end{align}
where  $\mathcal{M}_{f}^{N}({\|z\|_t})$ define in Lemma \ref{Lem-2.1} and $v_1: B_{{\ell}_t^n}\to B_{{\ell}_t^n}$  is a  {\it Schwarz mappings} having $z=0$ as a zero of order $m_1$ and  $R^p_{m_1,N}$ is the unique root in $(0,1)$ of the equation 
\begin{align}\label{BS-Eq-1.11}
	p\left(\dfrac{1-r^{m_1}}{1+r^{m_1}}\right)-\dfrac{2r^N}{1-r}=0.
\end{align}
The constant $R^p_{m_1,N}$ cannot be improved.
\end{thm}
\begin{rem}
	Theorem \ref{Thm-1.4} provides an extension of Theorem F for functions from $B_{{\ell}_t^n}$ to $\mathbb{D}^n$: inequality \eqref{Eq-2.6} {extends} inequality \eqref{Eq-1.3} of Theorem F when $m_1 = p = 1$, and {extends} inequality \eqref{Eq-1.4} of Theorem F when $m_1 = 1$ and $p = 2$.
\end{rem}
\section{Key lemmas and their proofs}
In this section, we present some necessary lemmas which will be used in proving our main results. Lemma A plays an important role in the proof of the Bohr–Rogosinski phenomena. 
\begin{lemA} \cite{Chen-Hamada-Ponnusamy-Vijayakumar-JAM-2024}
Suppose that $B_X$ and $B_Y$ are the unit balls of the complex Banach spaces $X$ and $Y$ , respectively. Let $f:B_X\to \overline{B_Y}$ be a holomorphic mapping. Then
\begin{align}
	{\|f(z)\|}_Y\leq \dfrac{{\|f(z)\|}_Y+{\|z\|}_X}{1+{\|f(z)\|}_Y{\|z\|}_X}\;\;\mbox{for}\;\;z\in B_X.
\end{align}
This estimate is sharp with equality possible for each value of ${\|f(z)\|}_Y$ and for each $z\in B_X.$
\end{lemA}
\begin{lemB}\cite[Lemma 3]{Lin-Liu-Ponnu-ACS-2021} For $p\in (0, 1]$ and $t\in [0, 1)$, we have
	\label{lem-power-estimate}
	\[ \frac{1-t^p}{1-t}\geq p.\]
\end{lemB}

\begin{lemC}\emph{(see \cite[Lemma 4]{Liu-Liu-Ponnusamy-BSM-2021})}
	If $ f(z)=\sum_{s=0}^{\infty}a_sz^s\in \mathcal{B} $, then for any $N\in\mathbb{N}$, the following inequality holds:
	\begin{align*}
		\sum_{s=N}^{\infty}|a_s|r^s&+\mathrm{sgn}(k)\sum_{s=1}^{k}|a_s|^2\dfrac{r^N}{1-r}+\left(\dfrac{1}{1+|a_0|}+\dfrac{r}{1-r}\right)\sum_{s=k+1}^{\infty}|a_s|^2r^{2s}\leq \dfrac{(1-|a_0|^2)r^N}{1-r}
	\end{align*}
	for $|z|= r\in[0,1),$  where $ k=\lfloor{(N-1)/2}\rfloor.$
\end{lemC}
To establish our main results of this paper, we prove Lemma C when  $f\in \mathcal{H}\left(B_{{\ell}_t^n},\overline{\mathbb{D}}^n \right)$ with series expansion given by \eqref{Eq-BS-1.6}.
\begin{lem} \label{Lem-2.1}
	Suppose $1\leq t\leq \infty,$ and  $f\in \mathcal{H}\left(B_{{\ell}_t^n},\overline{\mathbb{D}}^n \right)$ with series expansion given by \eqref{Eq-BS-1.6}. For $N\in \mathbb{N},$  $k=\lfloor{(N-1)/2}\rfloor$ and ${\|z\|}_t<1,$ we have
	 \begin{align*}
		&\mathcal{M}_{f}^{N}({\|z\|_t}):=\sum_{s=N}^{\infty}\dfrac{{\|D^sf(0)(z^s)\|}_\infty}{s!}+\mathrm{sgn}(k)\sum_{s=1}^{k}\left(\dfrac{{\|D^sf(0)(z^s)\|}_\infty}{s!}\right)^2\dfrac{{\|z\|}_t^{N-2s}}{1-{\|z\|}_t}\\&\nonumber+\left(\frac{1}{1+{\|f(0)\|_\infty}}+\frac{{\|z\|}_t}{1-{\|z\|}_t}\right)\sum_{s=k+1}^{\infty}\left(\dfrac{{\|D^sf(0)(z^s)\|}_\infty}{s!}\right)^2\leq \dfrac{(1-||a||_\infty^2){\|z\|}^N_t)}{1-{\|z\|}^2_t}. 
	\end{align*} 
	
\end{lem}
\begin{proof}[\bf Proof of Lemma \ref{Lem-2.1}]

	We fix $z\in \partial B_{{\ell}_t^n}\setminus\{ 0\}$ and $z_0=z/{||z||_t}.$ Then $z_0\in \partial B_{{\ell}_t^n}.$ Define $j$ such that $|z_j|=||z||_\infty=\max\{|z|_j:1\leq l\leq n\}.$ We define $h_j(\zeta)=f_j(\zeta z_0),$ $\zeta\in \mathbb{D}.$ Then $h_j\in \mathcal{B}$ and we have 
\begin{align*}
	h_j(\zeta)=a_j+\sum_{s=1}^{\infty}\dfrac{D^sf_j(0)(z_0^s)}{s!}\zeta^s,\;\; \zeta\in \mathbb{D}.
\end{align*}
Applying Lemma C for the function $h_j\in \mathcal{B}$,  we have
\begin{align*}
&\sum_{s=N}^{\infty}\dfrac{{|D^sf_j(0)(z_0^s)|}}{s!}|\zeta|^s+\mathrm{sgn}(k)\sum_{s=1}^{k}\left(\dfrac{{|D^sf_j(0)(z_0^s)|}}{s!}\right)^2\dfrac{{|\zeta|}^{N}}{1-{|\zeta|}}\\&\quad\nonumber+\left(\frac{1}{1+{|f_j(0)|}}+\frac{{|\zeta|}}{1-|\zeta|}\right)\sum_{s=k+1}^{\infty}\left(\dfrac{{|D^sf_j(0)(z_0^s)|}}{s!}|\zeta|^s\right)^2\leq \dfrac{(1-|a_j|^2){|\zeta|}^N}{1-{|\zeta|}^2}. 
\end{align*} 
  for all $j\in \{ 1,2,\dots, n\}.$\vspace{1.2mm}
  
  Set $|\zeta|={\|z\|}_t,$ we have
  \begin{align*}
  	&\sum_{s=N}^{\infty}\dfrac{{\|D^sf(0)(z^s)\|}_\infty}{s!}+\mathrm{sgn}(k)\sum_{s=1}^{k}\left(\dfrac{{\|D^sf(0)(z^s)\|}_\infty}{s!}\right)^2\dfrac{{\|z\|}_t^{N-2s}}{1-{\|z\|}_t}\\&\quad\nonumber+\left(\frac{1}{1+{\|f(0)\|_\infty}}+\frac{{\|z\|}_t}{1-{\|z\|}_t}\right)\sum_{s=k+1}^{\infty}\left(\dfrac{{\|D^sf(0)(z^s)\|}_\infty}{s!}\right)^2\leq \dfrac{(1-||a||_\infty^2){\|z\|}^N_t)}{1-{\|z\|}^2_t}. 
  \end{align*}
  This completes the proof.
\end{proof}
\section{Proofs of the main results}

	\begin{proof}[\bf Proof of Theorem \ref{Thm-4.1}]

	We fix $z\in \partial B_{{\ell}_t^n}\setminus\{ 0\}$ and $z_0=z/{||z||_t}.$ Then $z_0\in \partial B_{{\ell}_t^n}.$ Define $j$ such that $|z_j|=||z||_\infty=\max\{|z|_j:1\leq l\leq n\}.$ We define $h_j(\zeta)=f_j(\zeta z_0),$ $\zeta\in \mathbb{D}.$ Then $h_j\in \mathcal{B}$ and we have 
	\begin{align*}
		h_j(\zeta)=a_j+\sum_{s=1}^{\infty}\dfrac{D^sf_j(0)(z_0^s)}{s!}\zeta^s,\;\; \zeta\in \mathbb{D}.
	\end{align*}
	
	Then we have 
	 \begin{align*}
		\dfrac{|D^sf_j(0)(z_0^s)|}{s!}\leq (1-|a_j|^2)\;\;\mbox{for\; all}\;\; j=1,2,\dots n. 
	\end{align*}
	As $j$ is arbitrary, we have 
	\begin{align}\label{Eq-BS-1.12}
		\dfrac{{\|D^sf(0)(z_0^s)\|}_\infty}{s!}\leq (1-{{\|a\|}_\infty}^2)\,\;\mbox{for}\;\;s\in \mathbb{N}.
	\end{align}
	Let $b={\|a\|}_\infty\in [0,1).$
	For $r\in (0,1),$ by the estimates \eqref{Schwarz-k}, \eqref{Eq-BS-1.6} and \eqref{Eq-BS-1.12}, we have
	\begin{align}\label{Eq-4.3}
	{\|f(v_2(z))-f(0)\|}_\infty\leq \dfrac{(1-b^2)r^{m_2}}{1-r^{m_2}}.
	\end{align}
	For $r\in (0,1),$ by the estimates \eqref{Schwarz-k}, \eqref{Eq-BS-1.12}, \eqref{Eq-4.3} and Lemma A, we have
	\begin{align*}
		{\|f\left(v_1(z)\right)\|}_\infty^p&+\mu\sum_{s=1}^{\infty}	\frac{{\| D^{qs+K}f_{qs+K}(0)(z^{qs+K})\|}_\infty}{s!}
		+\nu {\|f(v_2(rz_0))-f(0)\|}_\infty
		\\
		&
		\leq  \left(\dfrac{b+r^{m_1}}{1+br^{m_1}}\right)^p+\mu(1-b^2)\dfrac{r^{q+K}}{1-r^q}
		+\nu(1-b^2) \dfrac{r^{m_2}}{1-r^{m_2}}
		\\
		&=1+\Pi_{p,q,K,m_1,m_2\mu,\nu}(b),
	\end{align*}
	where 
	\begin{align*}
		\Pi_{p,q,K,m_1,m_2,\mu,\nu}(b):=-1+\left(\dfrac{b+r^{m_1}}{1+br^{m_1}}\right)^p+\mu(1-b^2)\dfrac{r^{q+K}}{1-r^q}+\nu(1-b^2) \dfrac{r^{m_2}}{1-r^{m_2}}.
	\end{align*}
	We are taking $\varphi_0(r)=1$ and $N(r)=\mu r^{q+K}/(1-r^q)+\nu r^{m_2}/(1-r^{m_2})$ in \cite[Lemma 3]{Chen-Liu-Ponnusamy-RM-2023}, we obtain $\Pi_{p,q,K,m_1,m_2,\mu,\nu}(b)\leq 0$ for 
	$r\leq R_1.$
	Thus, we obtain
	\begin{align*}
		\left(\dfrac{b+r^{m_1}}{1+br^{m_2}}\right)^p+\mu(1-b^2)\dfrac{r^{q+K}}{1-r^q}+\nu(1-b^2) \dfrac{r^{m_2}}{1-r^{m_2}}\leq 1\;\;\mbox{for}\;\; r\leq R_1.
	\end{align*}
	Thus, the desired inequality obtained.\vspace{2mm}
	
	Next, we will show that the constant $R_1$ is optimal. For $b\in (0,1),$ let 
	\begin{align}\label{Eq-4.4}
		F(z)=\left(\frac{b+z_1}{1+bz_1}, 0,\dots, 0 \right),\;\; z=(z_1,\dots,z_n)\in  B_{{\ell}_t^n},
	\end{align}
	where $^\prime$ represent the transpose of the vector $z=\left( z_1, z_2,\dots, z_n\right)^\prime$ and $b\in (0,1).$ 
	Let $z_0\in \partial B_{{\ell}_t^n}$ and $z=(z_1,0,\dots,0)^\prime,$ which implies that ${\|z_0\|}_t=|z_1|=r.$
	Let $v_1(z)=l_{z_0}(z)^{m_1-1}z$ and $v_2(z)=l_{z_0}(z)^{m_2-1}z$ for $z\in B_{{\ell}_t^n}$.  According to the definition of Fr\'echet derivative , we have
	\begin{align*}
		DF(0)(z)=\left(\dfrac{\partial f_j(0)}{\partial z_i}\right)_{1\leq i,j\leq n}\left(z_1,z_2,\dots,z_n\right)^\prime.
	\end{align*}
	Since $z=(z_1,0,\dots,0)^\prime,$ we have $DF(0)(z)=\left(\frac{\partial f_1(0)}{\partial z_1}z_1, 0, \dots , 0\right),$ and therefore, ${\|DF(0)(z)\|}_\infty=\bigg|\frac{\partial f_1(0)}{\partial z_1}z_1\bigg|.$
	 With the help of the proof of \cite[Theorem 3.5]{Lin-Liu-Ponnusamy-Acta-2023}, we obtain
 for ${\|z\|}_t=r$,
	\begin{align}\label{BS-Eq-1.15}
	\dfrac{{\|D^sF(0)(z^s)\|}_\infty}{s!}=\bigg|\frac{\partial^s f_1(0)}{\partial z^s_1}\dfrac{z^s_1}{s!}\bigg|=(1-b^2)b^{s-1}r^s\;\; \mbox{for}\;\;s\in \mathbb{N}
	\end{align}
	and 
	\begin{align*}
		{\|F(v_2(z))-F(0)\|}_\infty
		=\frac{(1-b^2)r^{m_2}}{1-br^{m_2}}.
	\end{align*}
By a straightforward computations, we obtain for the function $F$ as follows 
	\begin{align}\label{Eq-4.5}
		&\nonumber{\|F\left(v_1(z)\right)\|}_\infty^p+\mu\sum_{s=1}^{\infty}	\frac{{\| D^{qs+K}F_{qs+K}(0)(z^{qs+K})\|}_\infty}{s!}
		+\nu {\|F(v_2(z))-F(0)\|}_\infty\\&\nonumber\leq
		\left(\dfrac{b+r^{m_1}}{1+br^{m_1}}\right)^p+\mu(1-b^2)\dfrac{b^{q+K-1}r^{q+K}}{1-b^qr^q}
		+\nu(1-b^2)\dfrac{r^{m_2}}{1-br^{m_2}}
		\nonumber \\&=1+(1-b)\mathcal{L}(b),	
	\end{align}
	where \[\mathcal{L}(b):= \dfrac{1}{1-b}\left(\left(\dfrac{b+r^{m_1}}{1+br^{m_1}}\right)^p-1\right)+\mu\dfrac{(1+b)b^{q+K-1}r^{q+K}}{1-b^qr^q}+\nu\dfrac{(1+b)r^{m_2}}{1-br^{m_2}}.\]
	Let $r\in (R_1,1)$ be arbitrary fixed. Since $R_1$ satisfies the equation $\Xi(r)=0,$ it follows that 
	\begin{align*}
		2\mu \frac{r^{q+K}}{1-r^q}+2\nu \frac{r^{m_2}}{1-r^{m_2}}-p\left(\frac{1-r^{m_1}}{1+r^{m_1}}\right)>0,\;\;\;\mbox{for}\;\; r\in (R_1,1).
	\end{align*}
	Consequently,
	\begin{align*}
		\lim\limits_{b\to 1^{-}}\mathcal{L}(b)=	2\mu \frac{r^{q+K}}{1-r^q}+2\nu \frac{r^{m_2}}{1-r^{m_2}}-p\left(\frac{1-r^{m_1}}{1+r^{m_1}}\right)>0,
	\end{align*}
	which implies that when $b\to 1^{-},$ the right hand side of the expression in 
	\eqref{Eq-4.5} is bigger than $1.$ This proves that $R_1$ cannot be improved. 
	This completes the proof.
	\end{proof}

\begin{proof}[\bf Proof of Theorem \ref{Thm-1.2}]
	Let $b={\|a\|}_\infty\in [0,1).$ In view of \eqref{Eq-4.3} and Lemma \ref{Lem-2.1} (for $N=1$), we have 
	\begin{align*}
		\mathcal{B}^p_f(r)&\leq b^p+\dfrac{(1-b^2)r}{1-r}+\dfrac{(1-b^2)r^{m_1}}{1-r^{m_1}}\\&=1+(1-b^2)\mathcal{K}(b), 
	\end{align*}
	where 
	\begin{align*}
		\mathcal{K}(b):=\dfrac{r}{1-r}+\dfrac{r^{m_1}}{1-r^{m_1}}-\left( \dfrac{1-b^p}{1-b^2}\right).
	\end{align*}
	Our aim is to show that $\mathcal{K}(b)\leq 0$ for each $b\in [0,1)$ and $r\leq R_2(p).$ Since $x\to \alpha(x):= (1-x^p)/(1-x^2)$ is decreasing function in $[0,1]$ for each $p\in (0,2],$ it follows that $\alpha(x)\geq\lim\limits_{x\to 1^-}\alpha(x)=p/2.$ Thus, $\mathcal{K}(b)$ is obviously an increasing function in $[0,1)$ and therefore, we have
	\begin{align*}
		\mathcal{K}(b)\leq \lim\limits_{b\to 1^-}\mathcal{K}(b)=\dfrac{r}{1-r}+\dfrac{r^{m_1}}{1-r^{m_1}}- \dfrac{p}{2}\leq 0\;\;\mbox{for}\;\; r\leq R_2(p).
	\end{align*}
	Hence, the desired inequality $\mathcal{B}^p_f(r)\leq 1$ holds for $r\leq R_2(p)$.\vspace{1.2mm}
	
	To prove the constant $R_2(p)$ is sharp for each $p\in (0,2],$ we consider the function $F$ given by \eqref{Eq-4.4}. Let $z_0\in \partial B_{{\ell}_t^n}$ and $z=(z_1,0,\dots,0)^\prime,$ which implies that ${\|z_0\|}_t=|z_1|=r.$
	Let $v_1(z)=l_{z_0}(z)^{m_1-1}z$ for $z\in B_{{\ell}_t^n}$. Then we have
	\begin{align*}
		\mathcal{B}^p_F(r)&=b^p+\dfrac{(1-b^2)r}{1-br}+\dfrac{(1-b^2)r^2}{(1+b)(1-r)(1-br)}+\dfrac{(1-b^2)r^{m_1}}{1-br^{m_1}}\\&=1+(1-b^2)\left(\dfrac{r}{1-r}+\dfrac{r^{m_1}}{1-r^{m_1}}-\left( \dfrac{1-b^p}{1-b^2}\right)\right).
	\end{align*}
Since,
	\begin{align*}
		r\to \Phi(r)=\dfrac{r}{1-r}+\dfrac{r^{m_1}}{1-r^{m_1}}-\left( \dfrac{1-b^p}{1-b^2}\right)
	\end{align*}
	is increasing in $(0,1),$  it is evident that $\Phi(r)>0$ in some interval $(R_2(p) R_2(p)+\epsilon)$. Hence, it is easy to see that when $b\to 1^-,$ the right side of the above expression is bigger than $1.$ This verifies that the constant $R_2(p)$ is best possible for each $p\in (0,2].$ This completes the proof. 
\end{proof}

\begin{proof}[\bf Proof of Theorem \ref{Thm-2.1}]
	Let $b={\|a\|}_\infty\in [0,1).$ In view of \eqref{BS-eq-1.3} and \eqref{Eq-BS-1.12},   we obtain for ${\|z\|}_t=r$ and $r\in (0,1),$ as follows
\begin{align}\label{BS-eq-2.7}
	S_z\leq  \dfrac{\left(1-{b}^2\right)^2r^2}{(1-r^2)^2}.
\end{align}
In view of  \eqref{BS-eq-2.7} and Lemma \ref{Lem-2.1}, by a simple computations for ${\|z\|}_t=r$, the following inequality can be obtained
\begin{align*}
	\mathcal{C}_f(r)&\leq 1+ p(b-1)+ \left(1-b^2\right)\frac{r}{1-r}+\sum_{s=1}^{N}d_s\left(\dfrac{(1-b^2)r}{1-r^2}\right)^{2s}\\&= 1+Q(b,r),
\end{align*}
where \begin{align*}
	Q(b,r):=\dfrac{(1-b^2)r}{1-r}+\sum_{s=1}^{N}d_s\left(\dfrac{(1-b^2)r}{1-r^2}\right)^{2s}-p(1-b).
\end{align*}
For all $b\in [0,1)$, by a straightforward computations, it can be shown that $Q(b,r)$ is a monotonically increasing function of $r$. Consequently, we obtain
\begin{align*}
	Q(b,r)\leq Q(b,p/(2+p))\;\;\mbox{for}\;\; b\in [0,1). 
\end{align*}
A straightforward calculation gives that
\begin{align*}
	Q(b,p/(2+p))=\dfrac{(1-b^2)}{2}\left(p+2F_N(b)-\dfrac{2p}{1+b}\right)=\dfrac{(1-b^2)}{2}\Phi(b),
\end{align*}
where \begin{align*}
	F_N(b):=\sum_{s=1}^{N}d_s(1-b^2)^{2s-1}\left(M_p\right)^{2s}\;\;\mbox{and}\;\; \Phi(b):=p+2F_N(b)-\dfrac{2p}{1+b}.
\end{align*}
To establish $Q(b,r)\leq 0,$ it suffices to show that $\Phi(b)\leq 0$ for $b\in [0,1].$ As $b\in [0,1 ]$, a simple calculation shows that
\begin{align*}
	b(1+b)^2\left(M_p\right)^2&\leq 4 \left(M_p\right)^2,\\ b(1+b)^2(1-b^2)^2\left(M_p\right)^4&\leq c_2\left(M_p\right)^4,\\ \vdots \\b(1+b)^2(1-b^2)^{2m-2}\left(M_p\right)^{2m}&\leq c_m\left(M_p\right)^{2m}.
\end{align*}
Thus, we see that 
\begin{align*}
	\Phi^{\prime}(b)&=\frac{2}{(1+b)^2}\bigg(p-2d_1b(1+b)^2\left(M_p\right)^2-6d_2b(1+b)^2(1-b^2)^2\left(M_p\right)^4-\cdots\\&\quad-2(2N-1)d_N(1+b)^2(1-b^2)^{2N-2}\left(M_p\right)^{2N}\bigg)\\&\geq \dfrac{2}{(1+b)^2} \left(p-\left(8d_1M^2_p+6c_2d_2M^4_p+\dots+2(2N-1)C_NM^{2N}_p\right)\right)\\&\geq 0,
\end{align*}
if  the coefficients $d_i$ of the polynomial $W_N$ satisfy the condition in \eqref{eq-2.2}. This indicates that $\Phi(b)$ behaves as an ascending function in $b\in [0,1]$, leading to the conclusion that $\Phi(b)\leq\Phi(1)=0$. This, in turn, establishes the desired inequality.\vspace{1.2mm}

To prove the constant $R_3$ is optimal, we consider the function $F$ given in \eqref{Eq-4.4}. In view of \eqref{BS-Eq-1.15}, it can be readily calculated that
\begin{align*}
	\mathcal{C}_F(r)=1-(1-b)\Psi_p^*(r),
\end{align*}
where 
\begin{align*}
	\Psi_p^*(r):=\dfrac{1-b^p}{1-b}-\dfrac{(1+b)r}{1-rb}-\dfrac{d_1r^2(1-b)(1+b)^2}{(1-b^2r^2)^2}-\dots-\dfrac{d_Nr^{2N}(1-b)^{2N-1}(1+b)^{2N}}{(1-b^2r^2)^{2N}}.
\end{align*}
For fixed $r>R_3=p/(2+p),$ we have 
\[
\lim\limits_{b\rightarrow 1^{-}}\Psi_p^*(r)=p-\frac{2r}{1-r}<0.
\]
Thus, it follows that $\Psi_p^*(r)<0$ for $b$ sufficiently close to $1$. 
Hence, we  conclude
\begin{align*}
	\mathcal{C}_F(r)=1-(1-b)\Psi_p^*(r)>1,
\end{align*}
which shows that the number $R_3$ is best possible.
\end{proof}

\begin{proof}[\bf Proof of Theorem \ref{Thm-1.4}]
		Let $b={\|a\|}_\infty\in [0,1).$	For $r\in (0,1),$ by the estimates \eqref{Schwarz-k}, Lemmas A and  D, we have
		\begin{align*}
			&{\|f(v_1(z))\|}_\infty^p+\mathcal{N}_{f}^{N}(z)\leq \left(\dfrac{b+r^{m_1}}{1+br^{m_1}}\right) ^p+\dfrac{(1-b^2)r^N}{1-r}=1+G^p_{m,N}(r),
		\end{align*}
	which is less than or equal to $1$ provided $G^p_{m,N}(r)\leq 0,$	where 
		\begin{align*}
			G^p_{m,N}(r):=\left(\dfrac{b+r^{m_1}}{1+br^{m_1}}\right) ^p-1+\dfrac{(1-b^2)r^N}{1-r}.
		\end{align*}
		Taking $\varphi_0(r)=1$ and $N(r)=r^N/(1-r)$ in \cite[Lemma 3]{Chen-Liu-Ponnusamy-RM-2023}, it can be easily shown that $ G^p_{m,N}(r) \leq 0$ for $r\leq R^p_{m_1,N}$, where $R^p_{m_1,N} $ is the unique positive root of the equation \eqref{BS-Eq-1.11}
	in $(0,1)$. Therefore, the desired inequality $\mathcal{D}_f(r)\leq 1$ holds for $\| z\|=r\leq R^p_{m_1,N}$.\vspace{2mm} 
	
	To prove the constant $R^p_{m_1,N}$ is best possible, 
	we consider the function $F$ given by \eqref{Eq-4.4}. Let $z_0\in \partial B_{{\ell}_t^n}$ and $z=(z_1,0,\dots,0)^\prime,$ which implies that ${\|z_0\|}_t=|z_1|=r.$ Let  $v_1(z)=l_{z_0}(z)^{m_1-1}z.$ In view of \eqref{BS-Eq-1.15}, by  straightforward calculations, we obtain 
	\begin{align}\label{BS-Eq-1.19}
			\nonumber\mathcal{D}_F(r)&=\left(\dfrac{b+r^{m_1}}{1+br^{m_1}}\right) ^p+\sum_{s=1}^{\infty}(1-b^2)b^{s-1}r^s+\mathrm{sgn}(k)\sum_{s=1}^{k}(1-b^2)^2b^{2s-2}\frac{r^N}{1-r} \\&\quad\nonumber+\left( \dfrac{1}{1+b}+\dfrac{r}{1-r}\right)\sum_{s=k+1}^{\infty}(1-b^2)^2b^{2(s-1)}r^{2s}\\& =1+(1-b)Q_{p,m,N}(r)
	\end{align}
	where \begin{align*}
		&Q_{p,m,N}(r):=\frac{1}{1-b}\left(\left(\dfrac{b+r^{m_1}}{1+br^{m_1}}\right) ^p-1\right)+\dfrac{(1+b)b^{N-1}r^N}{1-br}+\mathrm{sgn}(k)\frac{r^N}{1-r}\\&\times\sum_{s=1}^{k}(1+b)(1-b^2)b^{2s-2}+\left( \dfrac{1}{1+b}+\dfrac{r}{1-r}\right)\sum_{s=k+1}^{\infty}(1-b^2)(1+b)b^{2(s-1)}r^{2s}.
	\end{align*}
	For $r>R^p_{m_1,N}$ and choosing $b$ sufficiently closed to $1$ \emph{e.g.} $b\rightarrow 1^{-},$ we have
	\begin{align*}
		\lim\limits_{b\to 1^-}Q_{p,m,N}(r)=-p\left(\frac{1-r^{m_1}}{1+r^{m_1}}\right)+\frac{2r^N}{1-r}>0,
	\end{align*}
	which implies that the right hand side of the expression in 
\eqref{BS-Eq-1.19} is bigger than $1.$ This proves that $R^p_{m_1,N}$ is best possible. 
This completes the proof.
\end{proof}
\vspace{0.4 in }

\noindent{\bf Acknowledgment:} The research of the first author is supported by UGC-JRF (Ref. No. 201610135853), New Delhi, Govt. of India and second author is supported by SERB File No. SUR/2022/002244, Govt. of India.\vspace{1.5mm}

\noindent {\bf Funding:} Not Applicable.\vspace{1.5mm}

\noindent\textbf{Conflict of interest:} The authors declare that there is no conflict  of interest regarding the publication of this paper.\vspace{1.5mm}

\noindent\textbf{Data availability statement:}  Data sharing not applicable to this article as no datasets were generated or analysed during the current study.\vspace{1.5mm}

\noindent{\bf Code availability:} Not Applicable.\vspace{1.5mm}

\noindent {\bf Authors' contributions:} All the authors have equal contributions to prepare the manuscript.

\end{document}